\documentclass[12pt,a4paper]{article}
\usepackage[a4paper,top=2cm,bottom=2cm,left=2cm,right=2cm,bindingoffset=5mm]{geometry}
\usepackage[english,italian]{babel}
\usepackage[latin1]{inputenc}
\usepackage{amsfonts}
\usepackage{fancyhdr}
\usepackage{indentfirst}
\usepackage{color}
\usepackage{graphicx}
\usepackage{newlfont}

\usepackage{amssymb}
\usepackage{amsmath}
\usepackage{latexsym}
\usepackage{amsthm}
\usepackage{mathrsfs}
\usepackage{hyperref}
\hypersetup{colorlinks=true,linkcolor=blue}

\theoremstyle{plain}

\newtheorem{teo}{Theorem}[section]

\newtheorem{lemma}[teo]{Lemma}
\newtheorem{prop}[teo]{Proposition}
\newtheorem{corol}[teo]{Corollary}

\theoremstyle{definition}

\theoremstyle{remark}
\newtheorem{oss}[teo]{Remark}
{\left\lbrace\begin{array}{@{}l@{}}}%
{\end{array}\right.}

\newcommand{\R}{\mathbb{R}}

\newcommand{\elle}{\mathcal{L}}

\date{}
\begin{document}
\selectlanguage{english}%

\title{\bf{ON TWO-DIMENSIONAL NONLOCAL VENTTSEL' PROBLEMS IN PIECEWISE SMOOTH DOMAINS}}
\author{Simone Creo\thanks{Dipartimento di Scienze di Base e Applicate per l'Ingegneria, Universit\`{a} degli Studi di Roma "La Sapienza"
    Via A. Scarpa 16, 00161 Roma, Italy.
   E-mail: simone.creo@sbai.uniroma1.it, mariarosaria.lancia@sbai.uniroma1.it},
Maria Rosaria Lancia$^*$, Alexander Nazarov\thanks{St.~Petersburg Department of Steklov Mathematical Institute, Fontanka 27, 191023 St.~Petersburg, Russia, and St.~Petersburg State University, Universitetskii pr. 28, 198504 St.~Petersburg, Russia. E-mail: al.il.nazarov@gmail.com}, Paola Vernole\footnote{ Dipartimento di Matematica, Universit\`{a} degli Studi di Roma "La Sapienza",
    P.zale Aldo Moro 2,   00185 Roma, Italy. E-mail: vernole@mat.uniroma1.it}
		}
\maketitle

\begin{abstract}
\noindent We establish the regularity results for solutions of nonlocal Venttsel' problems in polygonal and piecewise smooth two-dimensional domains.
\end{abstract}

\bigskip



 \pagestyle{myheadings} \thispagestyle{plain}

\section*{Introduction}

\noindent In this paper we investigate an elliptic nonlocal Venttsel' problem for the Laplace operator in a bounded polygonal domain $\Omega\subset\R^2$.\\
Lately Venttsel' problems in irregular domains (for example having prefractal or fractal boundary) have been widely investigated, see e.g. \cite{JEE} and \cite{LRDV} and the references listed in. In \cite{JEE} the reader can also find motivations for the study of such problems.\\ 
There is a huge literature on {\it local} linear and quasi-linear Venttsel' problems (see e.g. \cite{AP-NAsurv}, \cite{AP-NA2}, \cite{nazarov}, \cite{Ar-Me-Pa-Ro}, \cite{Gol-Ru}, \cite{warma2009}, \cite{VELEZ2014} and the references listed in). As to the nonlocal case, among the others we refer to \cite{LVSV}, \cite{VELEZ2015}, \cite{WAR12} and the references listed in.\\
Our aim in this paper is to study the regularity in weighted Sobolev spaces of the weak solution of a nonlocal Venttsel' problem in a polygonal domain. These results will be crucial to obtain optimal a priori error estimates for the numerical approximation of the problem at hand; to this regard, for the local case, see \cite{Ce-Dl-La} and \cite{Ce-La-Hd}.\\
We first point out that a general nonlocal term appears also in the pioneering original paper by Venttsel' \cite{vent59}. Here we consider a nonlocal term which can be regarded as a version of the fractional Laplace operator $(-\Delta)^s$, for $0<s<1$, on the boundary. The presence of this  term could, in principle, deteriorate the regularity of the solution on the boundary. We prove that this is not the case, and that the weak solution of the nonlocal Venttsel' problem belongs to $H^2(\partial\Omega)$, i.e. it has the same regularity as in the local case (see \cite{Ce-Dl-La}).\\
It is well known that solutions of boundary value problems in piecewise smooth domains usually belong to weighted Sobolev spaces. In our case, the interplay between the boundary equation and the equation in the domain essentially influences the range of weight exponents, see \eqref{boundsigma}.\\
We remark that the techniques used in the local case to prove the regularity on the boundary are very different from the ones used in this paper.\\
The obtained results are a starting point in order to investigate the regularity of the solution of nonlocal Venttsel' problems in the case of domains with fractal boundary (for example of Koch-type domains).\\
The paper is organized as follows. In Section \ref{sec2} we define the domain and the functional spaces which will appear in this paper and we state the problem. In Section \ref{sec4} we prove a key a priori estimate for the solution. In Section \ref{sec5} we give an existence and uniqueness result for the weak and strong solutions of the nonlocal Venttsel' problem.

\section{Statement of the problem}\label{sec2}
\setcounter{equation}{0}

\noindent Let $\Omega\subset\R^2$ be a domain with polygonal boundary $\partial\Omega$. Namely, we suppose that $\partial\Omega$ is made by a finite number of segments, which form a finite number $N$ of angles $\alpha_j$, for $j=1,\dots,N$, and let us denote with $\alpha$ the opening of the largest angle in $\partial\Omega$.\\ 
In the following we denote by $L^2(\Omega)$ the Lebesgue space with respect to the Lebesgue measure $dx$ on $\Omega$, and by $L^2(\partial\Omega)$ the Lebesgue space on the boundary with respect to the arc length $d\ell$. By $H^s(\Omega)$, for $s\in\mathbb{N}$, we denote the standard Sobolev spaces. By $\mathcal{C}(\partial\Omega)$ we denote the set of continuous functions on $\partial\Omega$.\\
By $H^s(\partial\Omega)$, for $0<s<1$, we denote the Sobolev space on $\partial\Omega$ defined by local Lipschitz charts as in~\cite{necas}. For $s\geq 1$, we define the Sobolev space $H^s(\partial\Omega)$ by using the characterization given by Brezzi-Gilardi in~\cite{bregil}:
\begin{equation}\notag
H^s (\partial\Omega)=\{v\in \mathcal{C} (\partial\Omega)\,|\, v|_{\overset{\circ}{M}}\in H^s (\overset{\circ}{M})\},
\end{equation}
where $M$ denotes a side of $\partial\Omega$ and $\overset{\circ}{M}$ denotes the corresponding open segment (for the general case see Definition 2.27 in \cite{bregil}).\\
We denote the trace of $u$ on $\partial\Omega$ with $\gamma_0 u$. Sometimes we will use the same symbol to denote $u$ and its trace $\gamma_0 u$. The interpretation will be left to the context.

\noindent We now recall the Friedrichs inequality, see \cite[page 24]{mazya} for more details.
\begin{prop} Let $u\in H^1(\Omega)$. There exists a positive constant $C$ depending on $\Omega$ such that
\begin{equation}\label{poincare}
\|u\|^2_{L^2(\Omega)}\leq C\left(\|\nabla u\|^2_{L^2(\Omega)}+\|u\|^2_{L^2(\partial\Omega)}\right).
\end{equation}
\end{prop}

\noindent Let $r=r(x)$ be the distance from the set of vertices. For $\gamma\in\R$, and $s=1,2,\dots$, we denote by $H^s_\gamma (\Omega)$ the Kondratev (or weighted Sobolev) space of functions for which the norm
\begin{equation}\notag
\|u\|_{H^s_\gamma (\Omega)}=\left(\sum_{|k|\leq s} \int_{\Omega} r^{2(\gamma-s+|k|)} |D^{k} u(x)|^2\,dx\right)^{\frac{1}{2}}
\end{equation}
is finite, see \cite{kond}. For $s=0$, this space evidently coincides with the weighted Lebesgue space $L^2_\gamma (\Omega)$. We also define, for $s>0$ integer, the space $H^{s-\frac{1}{2}}_\gamma (\partial\Omega)$ as the trace space of $H^s_\gamma (\Omega)$ equipped with the norm
\begin{equation}\notag
\|u\|_{H^{s-\frac{1}{2}}_\gamma (\partial\Omega)}=\inf_{v=u\,\text{on}\,\partial\Omega}\,\|v\|_{H^s_\gamma (\Omega)}.
\end{equation}
We define the composite spaces
$$V^1(\Omega,\partial\Omega):=\{u\in H^1 (\Omega)\,:\, \gamma_0 u\in H^1 (\partial\Omega)\}$$
and
$$V^2_\sigma(\Omega,\partial\Omega):=\{u\in H^1 (\Omega)\,:\, r^\sigma D^2 u\in L^2(\Omega),\,\gamma_0 u\in H^2(\partial\Omega)\}.$$

\noindent We consider the problem formally stated as
\begin{eqnarray}\label{pbform}
& &-\Delta u=f \qquad\qquad\qquad\qquad\qquad\,\,\text{in $\Omega$,}\label{pbform1}\\[2mm]
& &-\Delta_\ell u=-\frac{\partial u}{\partial\nu}-bu-\theta_s(u)+g\quad \text{on $\partial\Omega$}\label{pbform2},
\end{eqnarray}
where $f$ and $g$ are given functions, $\displaystyle\Delta_\ell=\frac{\partial^2}{\partial\ell^2}$, $\nu$ the unit vector of exterior normal, $b\in L^\infty (\partial\Omega)$ and we set $\theta_s\colon H^{s} (\partial\Omega)\to H^{-s} (\partial\Omega)$ as follows: for every $u,v\in H^{s} (\partial\Omega)$
\begin{equation}\notag
\langle\theta_s (u),v\rangle=\iint_{\partial\Omega\times \partial\Omega}\frac{(u(x)-u(y))(v(x)-v(y))}{|x-y|^{1+2s}}\,d\ell(x)\,d\ell(y),
\end{equation}
where $\langle\cdot,\cdot\rangle$ denotes the duality pairing between $H^{-s} (\partial\Omega)$ and $H^{s} (\partial\Omega)$. We remark that the nonlocal term $\theta_s(\cdot)$ can be regarded as an analogue of the fractional Laplace operator $(-\Delta)^s$ on the boundary.\\
We now define the bilinear form as follows:
\begin{equation}
E(u,v)=\int_{\Omega} \nabla u\,\nabla v\,dx+\int_{\partial\Omega} \nabla_\ell u\,\nabla_\ell v\,d\ell+\int_{\partial\Omega} b\,u\,v\,d\ell+\langle\theta_s (u),v\rangle,
\end{equation}
for every $u,v\in V^1(\Omega,\partial\Omega)$.\\
\noindent We consider the weak formulation of problem \eqref{pbform1}-\eqref{pbform2}:
\begin{equation}\label{deb}
\begin{split}
&\text{Given $f$ and $g$, find $u\in V^1(\Omega,\partial\Omega)$ such that}\,\,\displaystyle E(u,v)=\int_{\Omega} f\,v\,dx+\int_{\partial\Omega} g\,v\,d\ell\\[2mm]
&\text{for every $v\in V^1(\Omega,\partial\Omega)$.}
\end{split}
\end{equation}
In what follows we denote by $C$ all positive constants. The dependence of constants on some parameters is given in parentheses. We do not indicate the dependence of $C$ on the geometry of $\Omega$.

\section{A priori estimates}\label{sec4}
\setcounter{equation}{0}

\begin{teo}\label{aprioriest} Let $u\in V^2_\sigma(\Omega,\partial\Omega)$ be a solution of problem \eqref{pbform1}-\eqref{pbform2}. Suppose that $s<3/4$. Then there exists a positive constant $C=C(\sigma)$ such that 
\begin{equation}\label{stima5}
\|u\|^2_{H^1(\Omega)}+\|r^\sigma D^2 u\|^2_{L^2(\Omega)}+\|u\|^2_{H^2(\partial\Omega)}\leq C(\sigma)(\|u\|^2_{L^2(\partial\Omega)}+\|r^\sigma f\|^2_{L^2(\Omega)}+\|g\|^2_{L^2(\partial\Omega)}),
\end{equation}
provided 
\begin{equation}\label{boundsigma}
1-\frac{\pi}{\alpha}<\sigma<\frac{1}{2},\qquad \sigma\geq-\frac{1}{2}
\end{equation}
(recall that $\alpha$ is the opening of the largest angle in $\partial\Omega$).
\end{teo}

\begin{proof}
We use the so-called \emph{Munchhausen trick}. 
We consider the right-hand side in \eqref{pbform2} as known functions. Then we easily have that 
\begin{equation}\label{stima1}
\|u\|^2_{H^2(\partial\Omega)}\leq C\left(\left\|\frac{\partial u}{\partial\nu}\right\|^2_{L^2(\partial\Omega)} + \|u\|^2_{L^2(\partial\Omega)}+\|\theta_s (u)\|^2_{L^2(\partial\Omega)}+\|g\|^2_{L^2(\partial\Omega)}\right).
\end{equation}
First we estimate $\|\theta_s(u)\|^2_{L^2(\partial\Omega)}$. Since $u\in H^2(\partial\Omega)$, it is sufficient to consider the local behavior of $u$ near the vertices. Without loss of generality, we can assume that the vertex is located at the origin. We introduce a smooth cutoff function $\eta$ and rectify $\partial\Omega$ near the origin. It is easy to see that $u\eta|_{\partial\Omega}$ becomes a function on $\R$ which is the sum of a smooth function and a term $c|t|\tilde\eta(t)$ (here $\tilde\eta$ is a one-dimensional cutoff function near the origin).\\
It is well known that $c|t|\tilde\eta(t)\in H^\beta (\R)$ for every $\beta<3/2$. This implies that $\theta_s (u)\in H^{\beta-2s} (\partial\Omega)$ and
\begin{equation}\notag
\|\theta_s(u)\|^2_{H^{\beta-2s}(\partial\Omega)}\leq C\|u\|^2_{H^2(\partial\Omega)},
\end{equation}
where $C$ depends on $\beta$ and $s$.\\
We fix $\beta\in(2s,3/2)$. From the compact embedding of $H^{\beta-2s} (\partial\Omega)$ in $L^2(\partial\Omega)$ we deduce that for every $\varepsilon>0$ there exists a constant $C(\varepsilon)$ such that
\begin{equation}\notag
\|\theta_s(u)\|^2_{L^2(\partial\Omega)}\leq\varepsilon\|\theta_s(u)\|^2_{H^{\beta-2s} (\partial\Omega)}+C(\varepsilon)\|\theta_s(u)\|^2_{H^{-s}(\partial\Omega)},
\end{equation}
see Lemma 6.1, Chapter 2 in~\cite{necas}. Similarly, we have
\begin{equation}\notag
\|\theta_s(u)\|^2_{H^{-s}(\partial\Omega)}\leq C\|u\|^2_{H^s(\partial\Omega)}\leq\varepsilon\|u\|^2_{H^2(\partial\Omega)}+C(\varepsilon)\|u\|^2_{L^2(\partial\Omega)}.
\end{equation}
Therefore we obtain the following estimate using \eqref{stima1}:
\begin{equation}\notag
\|u\|^2_{H^2 (\partial\Omega)} \leq C\left(\left\|\frac{\partial u}{\partial\nu}\right\|^2_{L^2(\partial\Omega)} +\|g\|^2_{L^2(\partial\Omega)}+\varepsilon\|u\|^2_{H^{2}(\partial\Omega)}+C(\varepsilon)\|u\|^2_{L^2(\partial\Omega)}\right).
\end{equation}
By choosing $\varepsilon$ sufficiently small we obtain
\begin{equation}\label{intermedia}
\|u\|^2_{H^2(\partial\Omega)}\leq C\left(\left\|\frac{\partial u}{\partial\nu}\right\|^2_{L^2(\partial\Omega)} + \|u\|^2_{L^2(\partial\Omega)}+\|g\|^2_{L^2(\partial\Omega)}\right).
\end{equation}
We now estimate $\left\|\frac{\partial u}{\partial\nu}\right\|^2_{L^2(\partial\Omega)}$. 
We consider a smooth function $U$ on $\overline\Omega$ which is linear near the corners of $\partial\Omega$ and such that $(u-U)(P)=\nabla_\ell(u-U)(P)=0$ in every vertex $P$ of $\partial\Omega$. Since $D^2 U$ vanishes in neighborhoods of vertices, without loss of generality we can assume that for every $\gamma\in\R$
\begin{equation}\label{stimaU}
\|U\|^2_{H^1(\Omega)}+\|r^{\gamma} D^2 U\|^2_{L^2 (\Omega)}+\|U\|^2_{H^2(\partial\Omega)}\leq C(\gamma)\|u\|^2_{H^2(\partial\Omega)}.
\end{equation}

If we consider the function $v=u-U$, from Hardy inequality applied on each segment of $\partial\Omega$ (see~\cite{hardy}) we obtain that $v\in H^2_{\gamma=0} (\partial\Omega)$. By rescaling we deduce $v\in H^\frac{3}{2}_{-\frac{1}{2}}(\partial\Omega)$, and
\begin{equation}\label{stimav}
\|v\|_{H^\frac{3}{2}_{-\frac{1}{2}}(\partial\Omega)}\leq C\|u\|_{H^2(\partial\Omega)}.
\end{equation}

Now we consider $v$ as the solution of the Dirichlet problem
\begin{equation}\label{problemav}
-\Delta v=f+\Delta U\in L^2_\sigma (\Omega);\qquad v|_{\partial\Omega}\in H^\frac{3}{2}_\sigma (\partial\Omega)
\end{equation}
(here we used the last restriction in \eqref{boundsigma}). From Theorem 3.1, Chapter 2 in~\cite{nazplam} (with $l=0$) it follows that $v\in H^2_{\sigma} (\Omega)$ if $|\sigma-1|<\pi/\alpha$ (we recall that $\alpha$ is the opening of the largest angle in $\partial\Omega$). 
With regard to \eqref{stimaU} and \eqref{stimav}, this implies
\begin{equation}\label{stimaU+v}
\|u\|^2_{H^1(\Omega)}+\|r^\sigma D^2 u\|^2_{L^2(\Omega)}\leq C(\sigma)(\|r^\sigma f\|^2_{L^2(\Omega)}+\|u\|^2_{H^2(\partial\Omega)})
\end{equation}
(to estimate the first term, we also used that $\sigma\leq 1$ in \eqref{boundsigma}).\\
By rescaling, we deduce that $\nabla u\in L^2_{\sigma-1/2} (\partial\Omega)$ and
\begin{equation}\label{stimaderivata}
\|\nabla u\|^2_{L^2_{\sigma-1/2} (\partial\Omega)}\leq\|u\|^2_{H^1(\Omega)}+\|r^\sigma D^2 u\|^2_{L^2(\Omega)}.
\end{equation}
We define a cutoff function $\eta_\delta$ such that
\begin{equation}\notag
\eta_\delta(r)=1\quad\text{for}\quad r>\delta,\qquad \eta_\delta (r)=0\quad\text{for}\quad r<\delta/2.
\end{equation}
Now we introduce the following trace operator:
\begin{equation}\notag
u\longrightarrow\frac{\partial u}{\partial\nu}\Big|_{\partial\Omega}=\eta_\delta\frac{\partial u}{\partial\nu}\Big|_{\partial\Omega}+(1-\eta_\delta)\frac{\partial u}{\partial\nu}\Big|_{\partial\Omega}=:\mathcal{K}_1(\delta)u+\mathcal{K}_2(\delta)u.
\end{equation}
The operator $\mathcal{K}_1(\delta)\colon H^2_\sigma(\Omega)\to L^2(\partial\Omega)$ is evidently compact. Using \eqref{stimaU+v}, we obtain for arbitrary $\varepsilon>0$
\begin{equation}\notag
\|\mathcal{K}_1(\delta)u\|^2_{L^2(\partial\Omega)}\leq\frac{\varepsilon}{2} (\|r^\sigma f\|^2_{L^2(\Omega)}+\|u\|^2_{H^2(\partial\Omega)})+C(\varepsilon,\sigma,\delta)\|u\|^2_{L^2(\partial\Omega)}.
\end{equation}
From \eqref{stimaU+v} and \eqref{stimaderivata} we deduce
\begin{equation}\notag
\|\mathcal{K}_2(\delta)u\|^2_{L^2(\partial\Omega)}\leq C(\sigma)\delta^{\frac{1}{2}-\sigma}(\|r^\sigma f\|^2_{L^2(\Omega)}+\|u\|^2_{H^2(\partial\Omega)}).
\end{equation}
By choosing $\delta(\sigma,\varepsilon)$ sufficiently small, we obtain
\begin{equation}\notag
\left\|\frac{\partial u}{\partial\nu}\right\|^2_{L^2(\partial\Omega)}\leq\varepsilon(\|r^\sigma f\|^2_{L^2(\Omega)}+\|u\|^2_{H^2(\partial\Omega)})+C(\varepsilon,\sigma)\|u\|^2_{L^2(\partial\Omega)}.
\end{equation}
Substituting the above inequality into \eqref{intermedia} we obtain
\begin{equation}\notag
\|u\|^2_{H^2(\partial\Omega)}\leq C\left(\varepsilon(\|r^\sigma f\|^2_{L^2(\Omega)}+\|u\|^2_{H^2(\partial\Omega)})+C(\varepsilon,\sigma)\|u\|^2_{L^2(\partial\Omega)}+\|g\|^2_{L^2(\partial\Omega)}\right).
\end{equation}
By choosing $\varepsilon$ sufficiently small we obtain
\begin{equation}\label{intermedia2}
\|u\|^2_{H^2(\partial\Omega)}\leq C\left(\|r^\sigma f\|^2_{L^2(\Omega)}+C(\sigma)\|u\|^2_{L^2(\partial\Omega)}+\|g\|^2_{L^2(\partial\Omega)}\right).
\end{equation}
Taking into account \eqref{stimaU+v}, we get the thesis.
\end{proof}

\section{Solvability of the Venttsel' problem}\label{sec5}
\setcounter{equation}{0}

\noindent We begin the existence and uniqueness of the weak solution.\\
By Friedrichs inequality (see \eqref{poincare}), we equip $V^1(\Omega,\partial\Omega)$ with the equivalent Hilbertian norm
\begin{equation}\notag
\|u\|_{V^1(\Omega,\partial\Omega)}=\left(\|\nabla u\|^2_{L^2(\Omega)}+\|\nabla_\ell u\|^2_{L^2(\partial\Omega)}+\|u\|^2_{L^2(\partial\Omega)}\right)^{\frac{1}{2}}.
\end{equation}
\begin{lemma}\label{lemma1} Let $b\geq 0$ and $b\not\equiv 0$. Then the energy form $E[u]=E(u,u)$ generates an equivalent norm in $V^1(\Omega,\partial\Omega)$.
\end{lemma}
\begin{proof} Since $b\in L^\infty(\partial\Omega)$ and
\begin{equation}\notag
\langle\theta_s (u),u\rangle\leq C\|u\|^2_{H^s (\partial\Omega)}\leq\,C\|u\|^2_{H^1(\partial\Omega)},
\end{equation}
we obtain that $E[u]\leq C\|u\|^2_{V^1(\Omega,\partial\Omega)}$. Then, since $\langle\theta_s(u),u\rangle\geq 0$, we have
\begin{equation}\notag
E[u]\geq\|\nabla u\|^2_{L^2(\Omega)}+\|\nabla_\ell u\|^2_{L^2(\partial\Omega)}.
\end{equation}
By the Poincar\'e inequality, $E[u]$ generates an equivalent norm on the subspace of functions in $V^1(\Omega,\partial\Omega)$ orthogonal to constants. Since the term $\int_{\partial\Omega}bu^2\,d\ell$ does not degenerate on constants, the statement follows.
\end{proof}


\noindent The following existence and uniqueness result holds.
\begin{corol}\label{weaksol} Let $f\in L^2(\Omega)$, $g\in L^2(\partial\Omega)$ and let $b$ be as in Lemma \ref{lemma1}. Then there exists a unique weak solution in $V^1(\Omega,\partial\Omega)$ of problem \eqref{deb}. Moreover
\begin{equation}\label{stimacont}
\|u\|_{V^1(\Omega,\partial\Omega)}\leq C(\|f\|_{L^2(\Omega)}+\|g\|_{L^2(\partial\Omega)}),
\end{equation}
where $C$ depends only on the coercivity constant of $E$.
\end{corol}

\noindent We finally prove the desired regularity for the weak solution of the nonlocal Venttsel' problem.

\begin{teo} Let $\sigma$ be subject to condition \eqref{boundsigma}. Suppose that $b$ satisfies the assumptions of Lemma \ref{lemma1}, $f\in L^2_\sigma (\Omega)$, $g\in L^2(\partial\Omega)$. Then the problem \eqref{pbform1}-\eqref{pbform2} has a unique solution $u\in V^2_\sigma (\Omega,\partial\Omega)$, and the following inequality holds
\begin{equation}\label{regolarita}
\|u\|^2_{H^1(\Omega)}+\|r^\sigma D^2 u\|^2_{L^2(\Omega)}+\|u\|^2_{H^2(\partial\Omega)}\leq C(\|r^\sigma f\|^2_{L^2(\Omega)}+\|g\|^2_{L^2(\partial\Omega)}),
\end{equation}
where $C$ depends on $\sigma$ and the coercivity constant of $E$.
\end{teo}

\begin{proof} We introduce the set of operators $\elle_\mu\colon V^2_\sigma(\Omega,\partial\Omega)\to L^2_\sigma (\Omega)\times L^2(\partial\Omega)$
$$\elle_\mu u:=\left(-\Delta u,\left(-\Delta_\ell u+bu+\mu\left(\frac{\partial u}{\partial\nu}+\theta_s(u)\right)\right)\Big|_{\partial\Omega}\right).$$

We claim that the operator $\elle_0$ is invertible. Indeed, it corresponds to the boundary value problem
$$-\Delta u=f \quad\,\text{in}\,\,\Omega,\qquad -\Delta_\ell u+bu=g\quad\,\text{on}\,\,\partial\Omega.$$
Here the equation in $\Omega$ and the boundary condition are decoupled. So we can first solve the boundary equation and then use its solution as the Dirichlet datum for the equation in the domain. The estimates similar to Theorem \ref{aprioriest} show that the solution belongs to $V^2_\sigma (\Omega,\partial\Omega)$ and inequality \eqref{regolarita} holds. So the claim follows.\\
The estimates in Theorem \ref{aprioriest} show that the operator
$$\elle_\mu-\elle_0\colon V^2_\sigma(\Omega,\partial\Omega)\to L^2_\sigma (\Omega)\times L^2(\partial\Omega);\qquad\elle_\mu u-\elle_0 u=\mu\left(0,\frac{\partial u}{\partial\nu}+\theta_s(u)\right)$$ is compact. Since ${\text{Ker}}(\elle_1)$ is trivial by Corollary \ref{weaksol}, the operator $\elle_1$ is also invertible, and the proof is complete.
\end{proof}

If $\Omega$ is a convex polygon, then $\alpha<\pi$. So we can put $\sigma=0$ and obtain the following result.

\begin{corol} Let $\Omega$ be a convex polygon. Suppose that $b$ satisfies the assumptions of Lemma \ref{lemma1}, $f\in L^2 (\Omega)$, $g\in L^2(\partial\Omega)$. Then the problem \eqref{pbform1}-\eqref{pbform2} has a unique solution $u\in H^2(\Omega)\cap H^2(\partial\Omega)$, and the following inequality holds
\begin{equation}\notag
\|u\|^2_{H^2(\Omega)}+\|u\|^2_{H^2(\partial\Omega)}\leq C(\|f\|^2_{L^2(\Omega)}+\|g\|^2_{L^2(\partial\Omega)}),
\end{equation}
where $C$ depends on the coercivity constant of $E$.
\end{corol}

If $\Omega$ is not convex, then $\pi<\alpha<2\pi$. In this case we obtain the following result.

\begin{teo} Let $\Omega$ be a non-convex polygon. Suppose that $b$ satisfies the assumptions of Lemma \ref{lemma1}, $f\in L^2 (\Omega)$, $g\in L^2(\partial\Omega)$. Then a unique solution of the problem \eqref{pbform1}-\eqref{pbform2} admits the following decomposition:
\begin{equation}
u(x)=\sum_{j\,:\,\alpha_j>\pi} c_j\chi (r_j) r^\frac{\pi}{\alpha_j}\sin(\pi\omega_j\alpha_j^{-1})+w(x).
\label{decomposizione}
\end{equation}
Here $(r_j,\omega_j)$ are local polar coordinates in a neighborhood of the angle with opening $\alpha_j$, $\chi$ is a cutoff function near the origin, and $w\in H^2(\Omega)\cap H^2(\partial\Omega)$. Moreover, the following inequality holds
\begin{equation}\notag
\|w\|^2_{H^2(\Omega)}+\|w\|^2_{H^2(\partial\Omega)}+\sum_{j\,:\,\alpha_j>\pi}|c_j|^2\leq C(\|f\|^2_{L^2(\Omega)}+\|g\|^2_{L^2(\partial\Omega)}),
\end{equation}
where $C$ depends on the coercivity constant of $E$.
\end{teo}

\begin{proof} Following the lines of the proof of Theorem \ref{aprioriest}, we obtain the Dirichlet problem for $v=u-U$
\begin{equation}\notag
-\Delta v\in L^2 (\Omega);\qquad v|_{\partial\Omega}\in H^\frac{3}{2} (\partial\Omega)
\end{equation}
instead of \eqref{problemav}. Theorem 3.4, Chapter 2 in \cite{nazplam} gives the representation \eqref{decomposizione} for $v$. Since $U$ is smooth, the statement follows. 
\end{proof}

\begin{oss} Without any sign condition on the coefficient $b$, the problem \eqref{pbform1}-\eqref{pbform2} is not necessarily solvable, but it has the Fredholm property.
\end{oss}

\begin{oss} All our results easily hold for an arbitrary piecewise smooth domain $\Omega\subset\R^2$ without cusps.
\end{oss}

\bigskip

\noindent {\bf Acknowledgements.} S. C., M. R. L. and P. V. have been supported by the Gruppo Nazionale per l'Analisi Matematica, la Probabilit\`a e le loro Applicazioni (GNAMPA) of the Istituto Nazionale di Alta Matematica (INdAM). A. N. was partially supported by Russian Foundation for Basic Research (RFBR) grant 15-01-07650.\\
This paper was completed during the visit of A. N. to Rome in February 2017. He would like to thank Universit\`a di Roma "Sapienza" for the hospitality.

\medskip


\end{document}